\documentclass{amsart}

\usepackage[english]{babel}
\usepackage[utf8]{inputenc}

\usepackage{mathrsfs}
\usepackage{mathtools}
\usepackage{dsfont}
\usepackage{amssymb}
\usepackage{amsmath}
\usepackage{enumerate}
\usepackage{amsfonts}
\usepackage{amsthm,amsmath}
\usepackage{bbm}
\usepackage{color}
\usepackage{hyperref}

\numberwithin{equation}{section}

\newcommand{\N}{\mathds{N}}
\newcommand{\Q}{\mathds{Q}}
\newcommand{\R}{\mathds{R}}

\newcommand{\cB}{\mathscr{B}}

\newcommand{\cL}{\mathscr{L}}
\newcommand{\cM}{\mathscr{M}}

\newcommand{\cT}{\mathscr{T}}

\newcommand{\cP}{\mathscr{P}}

\newcommand{\applied}[2]{\langle #1,#2\rangle}
\DeclarePairedDelimiter\norm{\lVert}{\rVert}
\DeclarePairedDelimiter\abs{\lvert}{\rvert}
\newcommand{\argument}{\,\cdot\,}
\renewcommand{\phi}{\varphi}

\newcommand{\regOps}{\cL^r}

\newcommand{\dx}{\;\mathrm{d}}
\newcommand{\loc}{\mathrm{loc}}

\theoremstyle{definition}
\newtheorem{definition}{Definition}[section]

\theoremstyle{plain}
\newtheorem{proposition}[definition]{Proposition}
\newtheorem{lemma}[definition]{Lemma}
\newtheorem{theorem}[definition]{Theorem}
\newtheorem{corollary}[definition]{Corollary}
\newtheorem*{theorem_no_number}{Theorem}

\begin{document}

\title{On Characteristics of the Range of Kernel Operators}
\author{Moritz Gerlach}
\address[M. Gerlach]{Universit\"at Potsdam, Institut f\"ur Mathematik, Karl--Liebknecht--Stra{\ss}e 24–25, 14476 Potsdam, Germany}
\email{gerlach@math.uni-potsdam.de}
\author{Jochen Gl\"uck}
\address[J. Gl\"uck]{Bergische Universit\"at Wuppertal, Fakult\"at f\"ur Mathematik und Naturwissenschaften, Gaußstra{\ss}e 20, 42119 Wuppertal, Germany}
\email{glueck@uni-wuppertal.de}
\date{\today}
\dedicatory{Dedicated to Justus}

\commby{}

\begin{abstract}
	We show that a positive operator between $L^p$-spaces is given by integration against a kernel function if and only if the image of each positive function
	has a  lower semi-continuous representative with respect to a suitable topology. 
	This is a consequence of a new characterization of kernel operators on general Banach lattices
	as those operators whose range can be represented over a fixed countable set of positive vectors.
	Similar results are shown to hold for operators that merely dominate a non-trivial kernel operator.
\end{abstract}

\maketitle

\section{Introduction} 
\label{sec:introduction}

The development of a representation theory for a certain operators on spaces of measurable functions in terms of integrals over measurable kernels
was essentially driven by Gelfand, Kantorovitch, Vulich, Dunford and Pettis \cite{dunford1940}.
Apparently, Nakano first recognized in \cite{nakano1953} the close relation between such operators and the band generated by finite rank operators.
This allows to transfer the concept of operators with integral kernels to the abstract concept of so-called \emph{kernel operators} on Banach lattices,
as it was done in the works of Lozanovski\'i \cite{lozanovskii1966} and Nagel and Schlotterbeck \cite{nagel1972}; see also \cite[IV \S 9]{schaefer1974}.
In \cite{bukhvalov1975} Bukhvalov stated that kernel operators can be described by a certain continuity condition, which has later
been called (sequential) star--order continuity.
This description was been proved independently by Bukhvalov \cite{bukhvalov1978} and Schep \cite{schep1979} 
(and is also contained in \cite[Ch~13, Sec~96]{zaanen1983}) 
and generalized to the
context of abstract vector lattices by Vietsch \cite{vietsch1979} and Grobler and van Eldik \cite{grobler1980}. 
As a consequence one obtains that every linear operator $T\colon L^p \to L^\infty$ (or, more generally, $T\colon L^p \to L^\infty_\loc$ \cite[Prop 1.7]{arendt1994}) 
on $\sigma$-finite measure spaces for some $1\leq p<\infty$ is a kernel operator \cite[Thm 98.2]{zaanen1983}.
In particular, every linear operator $T\colon L^p \to L^p$ that maps $L^p$-functions to continuous ones, is a kernel operator.

Nowadays, kernel operators appear naturally in a variety of ways. One example is the analysis of linear evolution equations: 
these equations often have smoothing properties which improve the regularity of solutions as time proceeds, 
and consequentially the operator semigroup that governs the equation consists of kernel operators. 

\subsection*{Contribution of this article}

As described above, it has been well known for many decades that a linear operator on an $L^p$-space that improves the 
regularity by mapping to bounded or continuous functions can be represented by an integral with a measurable kernel.
The contribution of this article is to complete the picture by providing a universal description of kernel operators by smoothing properties.
The main result for operators on $L^p$-spaces, proved in Theorem~\ref{thm:characterization-integral-operators}, is the following.

\begin{theorem_no_number}
	Let $(\Omega_1,\Sigma_1,\mu_1)$ and $(\Omega_2,\Sigma_2,\mu_2)$ be $\sigma$-finite measure spaces, $p\in [1,\infty)$ and $q\in [1,\infty]$.
	For each positive linear operator $T \colon L^p(\Omega_1,\Sigma_1,\mu_1) \to L^q(\Omega_2,\Sigma_2,\mu_2)$ the following are equivalent:
	\begin{enumerate}[\upshape (i)]
		\item The operator $T$ is a \emph{kernel operator}, 
		i.e.\ there exists a $\Sigma_1\otimes \Sigma_2$-measurable function $k\colon \Omega_1\times \Omega_2 \to \R_+$ such that
		for all $f\in L^p(\Omega_1,\Sigma_1,\mu_1)$ the following holds: 
		$k(\argument,y)f(\argument) \in L^1(\Omega_1,\Sigma_1,\mu_1)$ for $\mu_2$-almost all $y\in \Omega_2$ and 
		\begin{align*}
			Tf = \int_{\Omega_1} k(x, \argument) f(x) \dx \mu_1(x)
		\end{align*}
		$\mu_2$-almost everywhere.
	\item There exists a second countable topology $\cT$ on $\Omega_2$ with Borel $\sigma$-algebra $\cB(\cT) \subseteq \Sigma_2$
	such that $Tf$ has a lower semi-continuous representative with values in $[0,\infty]$ for every $f\in L^p(\Omega_1,\Sigma_1,\mu_1)_+$.
	\end{enumerate}
\end{theorem_no_number}

We emphasize that both implications in the theorem appear to be new. 
Note that the lower semi-continuous representatives of $Tf$ in (ii) need not be uniquely determined in general 
and that (ii) does not assert that they can be chosen in a linear way.
This distinguishes our implication ``(ii) $\Rightarrow$ (i)'' from a similarly looking but technically quite different implication in \cite[Cor~1.7]{jacob2006}:
as can be seen in the proof, \cite[Cor~1.7]{jacob2006} assumes that the lower semi-continuous representatives of $Tf$ do not only exist 
but can be chosen to depend linearly on $f$, which essentially means that the operator $T$ has a lifting that maps into the space of measurable functions.

It is worth noting that for every (arbitrarily non-regular) function $g\in L^p$ and each $h\in L^q$, where $p^{-1}+q^{-1}=1$, the rank-one operator $h\otimes g$, 
given by $f\mapsto \applied{h}{f}g$, is obviously a kernel operator on $L^p$. This is why, at a first sight, it seems hopeless to give a universal 
regularization property for general kernel operators: 
they are purely measure theoretical constructs and thus do not regard any given topology.
Instead, one has to choose a topology in dependence of the operator in order to obtain the equivalence.

The result can be derived from the following characterization of kernel operators on general Banach lattices, 
that we prove as Theorem~\ref{thm:kernel-characterization}. 
The general definition of \emph{kernel operators} on Banach lattices and its relation to integral kernels on $L^p$-spaces
is recalled in Section~\ref{sec:preliminaries}.

\begin{theorem_no_number}
	Let $E$ be a Banach lattice with order continuous norm such that $E_+$ contains a quasi-interior point. 
	Let $F$ be an order complete Banach lattice such that $F'$ contains an order continuous
	strictly positive functional. Then a positive linear operator $T\colon E\to F$ is a kernel operator if and only if
	there exists an at most countable set $H\subseteq F_+$ such that 
	\[ Tx = \sup \big([0,Tx]\cap \R_+H\big) \]
	for all $x\in E_+$.
\end{theorem_no_number}

Note that, by replacing $H$ with $\Q_+H$, the condition in the theorem can be replaced with the condition 
$Tx = \sup \big([0,Tx]\cap H\big)$. However, the formula in the theorem appears to be slightly more intuitive 
since Proposition~\ref{prop:quasi-units} below shows that one can, for instance, think of $H$ as a countable set of indicator functions.

Under different assumptions on the spaces, the implication ``$\Rightarrow$'' in the theorem 
can be derived from a characterization of kernel operators by means of equimeasurable sets, 
which is due to Schachermayer in the case of $L^p$-spaces \cite[Thm~4.4]{schachermayer1981}, 
due to Schep on more generally classes of function spaces \cite[Thm~3.3]{schep1981}
and due to Grobler and van Eldik in the setting of abstract vector lattices \cite[Thm~2.4]{grobler1985}. 

Loosely speaking, the theorem can be interpreted as follows: 
while on atomic Banach lattices (e.g.\ on $\ell^p$) every positive operator is a kernel operator, 
the theorem shows that in general positive kernel operators are described by the fact that the image of the positive cone can be constructed in a specific way from 
a fixed set of countably many positive vectors.
While kernel operators are known to factor through atomic spaces with regular factors, see e.g.\ the recent work \cite{blanco2020} and the references therein,
the theorem above gives a new and different perspective on similarities of general kernel operators and operators on atomic spaces.

In Section \ref{sec:partialkernels} we prove similar results for positive operators that merely dominate a non-zero kernel operator.
These \emph{partial kernel operators} play a useful role in the analysis of the asymptotic behavior of positive operator semigroups 
as shown for instance in \cite[Sec~4.2]{gerlach2019} and \cite{pichor2000}; see also \cite{gerlach2013b, gerlach2017}.  
For a certain class of operator semigroups on $L^1$-spaces the existence of a partial kernel operator within 
the semigroup even characterizes operator norm convergence of the semigroup as times tends to infinity, see \cite[Thm~1.1]{glueck2022}.

\section{Notation and terminology}
\label{sec:preliminaries}

Let $E,F$ be Banach spaces. We denote by $\cL(E,F)$ the space of bounded linear operators from $E$ to $F$.
The dual spaces of $E$ and $F$ are denoted by $E'$ and $F'$ and the adjoint of an operator $T \in \cL(E,F)$ by $T' \in \cL(F',E')$.

Throughout, we assume the reader to be familiar with the theory of Banach lattices.
Here, we only recall a few definitions and facts in order to fix the notation and the terminology. 
Beyond this, we refer to the classical monographs on Banach lattices such as, for instance, \cite{aliprantis19852006, meyer1991, schaefer1974}.

\subsection*{Banach lattices and positive operators} 

All Banach lattices in this article are real, i.e., the underlying scalar field is $\R$.
Let $E$ be a Banach lattice. A vector $x \in E$ is called \emph{positive} if $x \ge 0$ and the positive cone in $E$ is denoted by $E_+$.  
We write $x > 0$ if $x \ge 0$ but $x \not= 0$. For two vectors $x,z \in E$ we denote by
$[x,z] \coloneqq \{y \in E : x \le y \le z\}$ the \emph{order interval} between $x$ and $z$.
The dual space $E'$ of $E$ is a Banach lattice, too; 
a functional $\varphi\in E'$ satisfies $\varphi \in E'_+$ if and only if $\applied{\varphi}{x} \geq 0$ for all $x \in E_+$.
A functional $\varphi \in E'$ is called \emph{strictly positive} if $\applied{\varphi}{x} > 0$ for all $0 < x \in E$.

A vector subspace $J\subseteq E$ such that $\abs{x}\leq \abs{y}$ for $y\in J$ implies that $x\in J$ is
called an \emph{ideal}. Every ideal is a \emph{sublattice}, i.e., it is closed under the lattice operations.
For a given $q \in E_+$ the smallest ideal containing $q$ is called the \emph{principal ideal} generated by $q$ and is given by
$E_q = \{y \in E: \exists c \ge 0\text{, } \abs{y}\le cq\}$.
The so-called \emph{gauge norm} $\norm{\argument}_q$ on the principal ideal $E_q$ is defined as
\[ \norm{y}_q \coloneqq \min \{ c \geq 0 : \abs{y} \leq c q \}  \]
for each $y \in E_q$. It is a complete norm on $E_q$ that renders $E_q$ a Banach lattice; 
$(E_q, \norm{\argument}_q)$ is a so-called \emph{AM-space} with (order) unit $q$ \cite[Prop 1.2.13]{meyer1991}.
A Banach lattice $E$ whose norm satisfies $\norm{x+y}=\norm{x}+\norm{y}$ for all $x,y\in E_+$ is called an \emph{AL-space}.

If an ideal $J\subset E$ is closed under general suprema, i.e., $\sup A \in J$ for each set $A\subseteq J$ whose supremum exists in $E$,
then $J$ is called a \emph{band}. For each subset $A\subseteq E$ we denote by $A^\bot \coloneqq \{ y \in E : \abs{y}\wedge \abs{x} = 0 \text{ for all } x\in A\}$
the \emph{disjoint complement} of $A$ and we recall that $A^\bot$ is a band, and that $A^{\bot\bot} \coloneqq (A^\bot)^\bot$ is the smallest band containing $A$.
A vector $q \in E_+$ is called a \emph{weak unit} if the \emph{principal band} $\{q\}^{\bot\bot}$ is equal to the whole space $E$.
For a weak unit $q\in E_+$, we denote by 
\[
	Q(q) \coloneqq \{ v \in E_+ : v \wedge (q-v) = 0 \} = \{ v \in [0,q] : v \wedge (q-v) = 0 \}
\]
the set of all \emph{quasi-units} with respect to $q$.
If $E$ is an $L^p$-space over a $\sigma$-finite measure space for $p \in [1,\infty]$, then a vector $q \in E$ is a weak unit 
if and only if $q(\omega) > 0$ for almost all $\omega \in \Omega$; in this case, a vector $v \in L^p$ is a quasi-unit with respect to $q$ if and only if 
$v = q \mathds{1}_A$ for a measurable set $A$. 

Let $E$ and $F$ be Banach lattices. 
A linear operator $T\colon E \to F$ is called \emph{positive} if $TE_+ \subseteq F_+$. A positive linear operator mapping between two Banach lattices is automatically continuous \cite[Prop 1.3.5]{meyer1991}; by $\cL(E,F)_+$ we denote the collection of positive operators from $E$ to $F$.
We write $\regOps(E,F)$ for the regular operators from $E$ to $F$, i.e.\ the linear span of $\cL(E,F)_+$.

An operator $T\in \cL(E,F)$ satisfying $TE \subseteq F_q$ for a given $q\in F_+$ is also continuous from $E$ to $F_q$ 
when $F_q$ is endowed with the gauge norm $\norm{\argument}_q$; this follows from the closed graph theorem.

A linear projection $P\colon E\to E$ is called a \emph{band projection} if $0 \leq P \leq I$. 
In this case, $B\coloneqq PE$ is a \emph{projection band} and $I-P$ is the band projection onto $B^\bot$.
If every (principal) band in $E$ is a projection band, $E$ is said to have the \emph{(principal) projection property}.
The Banach lattice $F$ is said to be \emph{(countably) order complete} (or \emph{($\sigma$-)Dedekind complete}) 
if every (countable) order bounded subset of $F$ has a supremum and an infimum in $F$.
Every (countably) order complete vector lattice has the (principal) projection property \cite[Thm 1.2.9 and Prop 1.2.11]{meyer1991}.
Moreover, if $F$ is order complete, then $\regOps(E,F)$ is itself an order complete vector lattice \cite[Thm 1.3.2]{meyer1991}.

A net $(x_\lambda)_{\lambda\in \Lambda}$ in $E$ is called \emph{order convergent} to $x \in E$ if
there exists a decreasing net $(z_\lambda')_{\lambda'\in \Lambda'}$ in $E$ with $\inf z_{\lambda'} = 0$ such that for each $\lambda' \in \Lambda'$ 
there exists $\lambda_0 \in \Lambda$ such that $\abs{x_\lambda-x}\leq z_{\lambda'}$ for every $\lambda \geq \lambda_0$;
this can easily be seen to be equivalent to the following property:
there exists a non-empty set $S \subseteq E_+$ with infimum $0$ and such that
for every $s \in S$ there exists $\lambda_0 \in \Lambda$ with the property $\abs{x_\lambda - x} \le s$ for all $\lambda \ge \lambda_0$.
An operator $T\in \regOps(E,F)$ is \emph{order continuous} if $(Tx_\lambda)$ order converges to $Tx$ for every net $(x_\lambda)$ in  $E$ 
that order converges to a point $x \in E$. 

The Banach lattice $E$ is said to have \emph{order continuous norm}
if $\inf \norm{x_\alpha}=0$ for every decreasing net $(x_\alpha)\subseteq E_+$ with $\inf x_\alpha =0$.
For further important properties of these spaces we refer to \cite[Sec II.5]{schaefer1974} or \cite[Sec 2.4]{meyer1991}.
It is worthwhile pointing out that every $L^p$-space for $p\in [1,\infty)$ (over an arbitrary measure space) 
has order continuous norm.
If the norm on $E$ is order continuous, then $E$ is in particular order complete \cite[Thm 2.4.2]{meyer1991} 
and every regular operator from $E$ to $F$ is order continuous.

\subsection*{Kernel operators} 
Let $E$ and $F$ be Banach lattices such that $F$ is order complete.  We denote by $E'\otimes F$ the space of all finite rank operators from $E$ to $F$.
The elements of $(E'\otimes F)^{\bot\bot}$, the band generated by $E'\otimes F$ within the vector lattice $\regOps(E,F)$, are called \emph{kernel operators}.
We call a regular operator $T\in\regOps(E,F)$ a \emph{partial kernel operator} 
if either $T$ is zero or there exists 
a positive and non-zero kernel operator $K\in (E'\otimes F)^{\bot\bot}$ such that $\abs{T} \geq K$.

While the above definition of kernel operators is rather abstract, 
there exists a more explicit description on concrete function spaces.
In fact, order continuous positive kernel operators on $L^p$-spaces can be represented by integral kernels as follows:

Let $p,q \in [1,\infty]$, let $(\Omega_1,\Sigma_1,\mu_1)$ and $(\Omega_2,\Sigma_2,\mu_2)$ be $\sigma$-finite measure spaces
and let $(\Omega,\Sigma,\mu)$ denote their product space. 
A positive linear operator
\[ T\colon L^p(\Omega_1,\Sigma_1,\mu_1)\to L^q(\Omega_2,\Sigma_2,\mu_2) \]
is an order continuous kernel operator if and only if there exists a measurable function $k \colon  \Omega \to \R_+$ 
such that for every $f\in L^p(\Omega_1,\Sigma_1,\mu_1)$ the following holds: 
for $\mu_2$-almost every $y\in \Omega_2$  
the function $f(\argument)k(\argument,y)$ is in $L^1(\Omega_1,\Sigma_1,\mu_1)$ and the equality
\begin{align}
	\label{eqn:integralop}
	Tf = \int_{\Omega_1} f(x)k(x,\argument) \dx \mu_1(x).
\end{align}
holds in $L^q(\Omega_2,\Sigma_2,\mu_2)$. 
For this equivalence we refer to \cite[Prop~IV.9.8]{schaefer1974}.
In this case the function $k$ (which is uniquely determined up to a nullset) 
is called the \emph{integral kernel} of $T$.

\section{The Range of Kernel Operators}
\label{sec:range-of-kernel-ops}

Throughout this section let $E$ and $F$ be Banach lattices.
We start with a simple auxiliary result about order continuous functionals.

\begin{lemma}
	\label{lem:order-continuous-separating}
	Assume that $F$ is order complete and that $F'$ contains a strictly positive order continuous functional. 
	Then the order continuous functionals in $F'$ separate $F$, 
	i.e., for every non-zero $y \in F$ there exists an order continuous functional $\psi \in F'$ 
	such that $\langle \psi, y \rangle \not= 0$.
\end{lemma}

\begin{proof}
	Fix a strictly positive order continuous functional $\varphi \in F'$ and let $y \in F$ be non-zero. 
	As $F$ is order complete, there exist band projections $P,Q$ on $F$ 
	onto the bands generated by $y^+$ and $y^-$, respectively.
	Since the order continuous functionals are a band in $F'$ \cite[Cor on p.\,74]{schaefer1974}, $P'\phi$ and $Q'\phi$ are 
	order continuous. 
	As $y\neq 0$ and $\phi$ is strictly positive,
	it follows that $\applied{P'\phi}{y} = \applied{\phi}{y^+}>0$ or $\applied{Q'\phi}{y}=\applied{\phi}{y^-}>0$.
%
%
%
\end{proof}

We are now going to derive one implication of our (abstract) main result from a characterization of kernel operators by means of nuclear operators 
that goes back to Nagel and Schlotterbeck (see \cite[Satz~5]{nagel1972} or, for an English reference, \cite[Theorem~IV.9.6]{schaefer1974}).

\begin{theorem}
	\label{thm:kernel-approx-non-sep}
	Assume that $E_+$ contains a quasi-interior point, 
	let $F$ be order complete, and assume that $F'$ contains a strictly positive order continuous functional.
	Then for every positive kernel operator $T \colon E \to F$ 
	there exists an at most countable set $H \subseteq F_+$ such that for all $x\in E_+$
	\[ Tx = \sup \big([0,Tx]\cap \R_+H \big).\]
\end{theorem}

\begin{proof}
	Let $z \in E_+$ be a quasi-interior point 
	and let $\varphi \in F'$ be a strictly positive order continuous functional. 
	We will first show that there exists an at most countable set $H \subseteq [0,Tz] \subseteq F_+$ 
	such that $Tx = \sup \big([0,Tx]\cap H \big)$ for all $x \in [0,z]$.
	
	To this end we endow the principal ideal $E_z$ with its gauge norm, which turns $E_z$ into an AM-space.
	Moreover, we denote the completion of $F$ with respect to the norm $y \mapsto \applied{\phi}{\abs{y}}$ by $F^\varphi$. 
	Then $F^\varphi$ is an AL-space. Since $\varphi$ is order continuous and $F$ is order complete, $F$ is an ideal in $F^\varphi$ 
	\cite[bottom of p.\,243]{schaefer1974}. 
	According to Lemma~\ref{lem:order-continuous-separating} the order continuous functionals in $F'$ separate $F$, 
	so the Nagel--Schlotterbeck characterization of kernel operators 
	\cite[Thm~IV.9.6]{schaefer1974} is applicable and yields that the composition operator
	\begin{align*}
		 E_z \hookrightarrow E \overset{T}{\to} F \hookrightarrow F^\varphi
	\end{align*}
	is nuclear, where the mappings on the left and on the right are the canonical injections. 
	This means that there exist sequences $(\alpha_n)$ in $(E_z)'$ and $(y_n)$ in $F^\varphi$ 
	such that $\sum_n \norm{\alpha_n} \norm{y_n} < \infty$ and 
	\begin{align*}
		T x = \sum_n \applied{\alpha_n}{x} y_n
	\end{align*}
	for each $x \in E_z$. Note that the series converges absolutely in $F^\varphi$. 
	By splitting each $\alpha_n$ and each $y_n$ into positive and negative part and regrouping the series, 
	we find sequences $(\beta_n)$, $(\gamma_n)$ in $(E_z)'_+$ and $(v_n)$, $(w_n)$ in $(F^\varphi)_+$ 
	such that $\sum_n \norm{\beta_n} \norm{v_n} < \infty$ and $\sum_n \norm{\gamma_n} \norm{w_n} < \infty$ and 
	\begin{align*}
		Tx = \sum_n \applied{\beta_n}{x} v_n - \sum_n \applied{\gamma_n}{x} w_n
	\end{align*}
	for all $x \in E_z$. Again, both series converge absolutely in $F^\varphi$. 
	Now let $H_1$ denote the at most countable set of all vectors in $(F^\phi)_+$ of the form 
	\begin{align*}
		\sum_{n = 1}^N q_n v_n,
	\end{align*}
	where $N$ runs through $\N$ and the $q_n$ run through $\Q$. 
	Similarly, let $H_0$ denote the at most countable set of all vectors in $(F^\varphi)_+$ of the form 
	\begin{align*}
		\sum_{n = 1}^N q_n w_n  +  \sum_{n=N+1}^\infty \applied{ \gamma_n}{ z }w_n,
	\end{align*}
	where again $N$ runs through $\N$ and the $q_n$ run through $\Q$.
	Finally, define
	\begin{align*}
		H \coloneqq \big\{ (h_1 - h_0)^+ \wedge Tz :  h_1 \in H_1 \text{ and } h_0 \in H_0 \big\}.
	\end{align*}
	Since $Tz \in F_+$ and $F$ is an ideal in $F^\varphi$, one has $H\subseteq [0,Tz] \subseteq F_+$.
	
	Now let $x \in [0,z]$. 
	Then there exists a sequence $(h_{n,1})$ in $H_1$ which converges, with respect to the norm in $F^\varphi$, to 
	\begin{align*}
		\sum_n \applied{ \beta_n}{ x } v_n,
	\end{align*}
	and which is dominated by this limit. 
	Similarly, due to the choice of the elements of $H_0$, 
	there exists a sequence $(h_{n,0})$ in $H_0$ which converges, again with respect to the norm in $F^\varphi$, to 
	\begin{align*}
		\sum_n \applied{ \gamma_n }{ x } w_n,
	\end{align*}
	and which dominates this limit. 
	Therefore, 
	$(h_{n,1} - h_{n,0})$ is a sequence in $H_1-H_0$ that converges to $Tx$ with respect to the norm on $F^\varphi$
	and which is order bounded above by $Tx$. 
	This implies that the sequence of elements $(h_{n,1} - h_{n,0})^+ \wedge Tz\in H$ is contained in $[0,Tx]$ and converges 
	to $Tx$ with respect to the norm on $F^\varphi$. 
	So $Tx$ is the supremum of $[0,Tx] \cap H$ within the space $F^\varphi$ and thus, in particular, within the space $F$, 
	as claimed at the beginning of the proof.

	This readily implies that $Tx = \sup \big([0,Tx]\cap \R_+ H \big)$ for each $0 \le x \in E_z$. 
	That the same equation even holds for every $0 \le x \in E$ follows from the fact that $x \wedge (nz) \uparrow x$ in $E$-norm for each $0 \le x \in E$
	since $z$ is a quasi-interior point of $E_+$
	\cite[Thm~II.6.3]{schaefer1974}.
\end{proof}

The asssumption that $E_+$ contains a quasi-interior point cannot be dropped in Theorem \ref{thm:kernel-approx-non-sep}.
Indeed, if $X$ is an uncountable set, then the identity operator on the space $E \coloneqq F \coloneqq \ell^1(X)$ 
is a kernel operator and the norm functional on $F$ is strictly positive and order continuous.
However, the conclusion of Theorem~\ref{thm:kernel-approx-non-sep} fails. 
The only assumption that is not satisfied is the existence of a quasi-interior point in $E_+$.

For the case of $L^p$-spaces, it is natural to ask whether one can choose 
the functions in $H$ in Theorem~\ref{thm:kernel-approx-non-sep} as indicator functions.
The following proposition, a consequence of Freudenthal's spectral theorem, 
shows that this is possible.

\begin{proposition}
	\label{prop:quasi-units}
	Let $F$ be countably order complete,
	let $R \subseteq F_+$, and assume that there exists an at most countable set $H \subseteq F_+$ 
	such that $y = \sup\big( [0,y] \cap \R_+ H \big)$ for all $y \in R$. 
	If $q \in F_+$ is a weak order unit, then there exists an at most countable set of quasi-units $U \subseteq Q(q)$ 
	such that $y = \sup\big( [0,y] \cap \R_+ U \big)$ for all $y \in R$.
\end{proposition}

\begin{proof}
	Every $h \in H$ is the supremum of $\{ h \wedge nq : n\in \N\}$, see \cite[Prop~II.2.11]{schaefer1974}. 
	Hence, when replacing  $H$ with $\{h \wedge nq  : h \in H \text{, }n \in \N\}$, 
	the formula $y = \sup\big( [0,y] \cap \R_+ H \big)$ remains valid for all $y \in R$. 
	Thus we may, and shall, assume that $H$ is contained in the principal ideal $F_q$. 
	As $F$ is countably order complete, it has the principal projection property \cite[Cor~2 on p.\,64]{schaefer1974}. 
	So Freudenthal's spectral theorem \cite[Thm~40.2]{luxemburg1971} is applicable
	and implies that each $h \in H$ is the supremum of an at most countable subset $U_h \subseteq Q(q)$. 
	Hence, the set $U \coloneqq \cup_{h \in H} U_h$ has all the required properties.
\end{proof}

Our next goal is to show that, on many spaces, the approximation obtained by 
Theorem~\ref{thm:kernel-approx-non-sep} even characterizes positive 
kernel operators (Theorem~\ref{thm:q-cont-approx-kernel-operators}).
For the proof we use the following auxiliary result.

\begin{lemma}
	\label{lem:grobler-van-eldik}
	Assume that $E$ has order continuous norm and that $F$ is order complete.
	If $T \in \cL(E,F)_+$ is norm--order continuous on order bounded sets, then $T$ is a kernel operator.
\end{lemma}

Here, \emph{norm--order continuous on order bounded sets} means that for every order bounded net $(x_\alpha)$ in $X$ that norm converges to $x \in X$,
the net $(Tx_\alpha)$ in $F$ order converges to $Tx$.
It is not difficult to see that this is equivalent to the formally weaker property that for every $x \in E_+$ and every net $(x_\alpha)$ in $[0,x]$ 
that norm converges to $x$, the net $(Tx_\alpha)$ order converges to $Tx$.

The result in Lemma~\ref{lem:grobler-van-eldik} is essentially known (under slightly different and more technical assumptions) 
and can be adapted to yield even a characterization of regular kernel operators \cite[Thm~3.9]{grobler1980}, 
which is an abstract version of Bukhvalov's characterization of kernel operators by star--order continuity \cite{bukhvalov1975}.
However, under the assumptions made in the lemma it is possible to give a particularly easy and transparent proof:

\begin{proof}[Proof of Lemma~\ref{lem:grobler-van-eldik}]
	Let $S \in \cL(E,F)_+$ be the largest kernel operator that is dominated by $T$ and fix $x \in E_+$.
	We show that $Sx = Tx$.
	
	Consider the set
	\begin{align*}
		\Lambda = \{(\varepsilon, y) \in (0,\infty) \times [0,x] :  \norm{x - y} < \varepsilon\},
	\end{align*}
	which is rendered a directed set by the relation $\preceq$ given by $(\varepsilon_1, y_1) \preceq (\varepsilon_2, y_2)$ if and only $\varepsilon_2 \le \varepsilon_1$. 
	Then $(y)_{(\varepsilon, y) \in \Lambda}$ is a net in $[0,x]$ that norm converges to $x$ and hence, $(Ty)_{(\varepsilon, y) \in \Lambda}$ order converges to $Tx$. 
	So there exists an increasing net $(z_\alpha)_{\alpha \in A}$ in $F$ with supremum $Tx$ such that for each $\alpha \in A$ 
	there is $\varepsilon \in (0,\infty)$ such that $Ty \ge z_{\alpha}$ for each $y \in [0,x]$ that is closer than $\varepsilon$ to $x$. 
	Now, fix $\alpha \in A$ and choose $\varepsilon \in (0,\infty)$ with the aforementioned property.
	It suffices to prove that $Sx \ge z_{\alpha}$.

	As $E$ has order continuous norm there exists a functional $\phi \in E'_+$ which is strictly positive on $[0,x]$ \cite[Thm 4.15]{aliprantis19852006}.
	Moreover, also due to the order continuity of the norm on $E$, 
	by rescaling $\phi$ one can achieve that $\applied{\phi}{y} \ge 1$ for all $y \in [0,x]$ of norm $\norm{y} \ge \varepsilon$ \cite[Lem 4.16]{aliprantis19852006}. 
	This implies that
	\begin{align*}
		Sx & \ge \big(T \wedge (\phi \otimes z_{\alpha})\big) x \\
		& = \inf \{Ty_1 + \applied{\phi}{y_2}z_{\alpha} :  y_1,y_2 \in [0,x] \text{, } y_1 + y_2 = x\} \ge z_\alpha,
	\end{align*}
	since either $\norm{y_2} \ge \varepsilon$ and thus $\applied{\phi}{y_2} \ge 1$, or $\norm{x-y_1} = \norm{y_2} < \varepsilon$ and thus $Ty_1 \ge z_{\alpha}$.
\end{proof}

\begin{theorem}
	\label{thm:q-cont-approx-kernel-operators}
	Let $E$ have continuous norm and let $F$ be order complete such that $F'$ contains a strictly positive order continuous functional.
	Let $T\in \cL(E,F)_+$ and assume that there exists an at most countable set $H\subseteq F_+$ such that 
	$Tx = \sup ([0,Tx]\cap \R_+ H)$ for every $x\in E_+$. 
	Then $T$ is a kernel operator.
\end{theorem}

\begin{proof}
	By replacing $H$ with $\Q_+ H$ we may and shall assume that even $Tx = \sup ([0,Tx]\cap H)$ for every $x\in E_+$. 
	Moreover, there is no loss of generality in assuming that $\norm{T} = 1$.
	Finally we may, and will, assume that $H$ is closed under finite suprema and thus upwards directed.
	Let $\Phi \in F'$ be strictly positive, order continuous and of norm $\norm{\Phi}=1$.	
	
	Fix $x \in E_+$ and let $(x_\lambda)_{\lambda \in \Lambda}$ be a net in $[0,x]$ that is norm convergent to $x$.
	According to Lemma~\ref{lem:grobler-van-eldik} and the discussion after the lemma it suffices to show that 
	$(Tx_\lambda)_{\lambda \in \Lambda}$ order converges to $Tx$.
	To this end, we will now show that for each $\varepsilon > 0$ there exist a number $\delta > 0$ and a vector 
	$z_\varepsilon \in [0,Tx]$ with the following properties:
	\begin{enumerate}[(a)]
		\item
		For each $y \in [0,x]$ that satisfies $\norm{x-y} \le \delta$ we have $Ty \ge z_\varepsilon$.
		
		\item 
		One has $\applied{\Phi}{Tx - z_\varepsilon} \le 2\varepsilon$.
	\end{enumerate}
	Property~(b) then implies that $\sup_{\varepsilon > 0} z_\varepsilon = Tx$ and property~(a) implies that, 
	for each $\varepsilon > 0$, one has $Tx_\lambda \ge z_\varepsilon$ for all sufficiently large $\lambda$;
	thus, $(Tx_\lambda)_{\lambda \in \Lambda}$ order converges to $Tx$.
	
	So fix $\varepsilon > 0$.
	Let us consider the subset 
	\[ M \coloneqq  \bigl\{ y \in [0,x] : \norm{x-y} < \varepsilon \bigr\} \]
	of $E$; it is an open subset of the topological space $[0,x]$ and thus a Baire space.
	
	We claim that for each $y \in M$ we can find a vector $z \in [0,Ty] \cap H$ that satisfies $\applied{\Phi}{Tx - z} \leq 2 \varepsilon$.
	Indeed, fix $y \in M$. We have $Ty = \sup([0,Ty] \cap H)$ by assumption, so since $H$ is closed under finite suprema 
	and $\Phi$ is order continuous, there exists a vector $z \in [0, Ty]\cap H$ that satisfies $\applied{\Phi}{Ty-z} \le \varepsilon$ and thus,
	\[ \applied{\Phi}{Tx-z} = \applied{\Phi}{Tx-Ty} + \applied{\Phi}{Ty-z} \leq 2\varepsilon.  \]
	In other words,
	\[ M \subseteq \bigcup_{z \in Z} \big\{y \in M : Ty \ge z \big\}, \]
	where 
	\begin{align*}
		Z \coloneqq \{z \in [0,Tx] \cap H: \applied{\Phi}{Tx-z} \le 2\varepsilon \}
	\end{align*}
	As $Z$ is at most countable and $M$ is a Baire space, there exists a vector $z_\varepsilon \in Z$ such 
	that the set $\big\{y \in M : \; Ty \ge z_\varepsilon \big\}$ has non-empty interior within $M$.
	
	Obviously, $z_\varepsilon$ satisfies (b), and in order to show that it also satisfies (a) we 
	choose a point $y_0$ in the interior of $\big\{y \in M :  Ty \ge z_\varepsilon \big\}$ within $M$.
	Let $\delta > 0$ such that all $\tilde y \in M$ that satisfy $\norm{y_0 - \tilde y} \le \delta$ are 
	also in $\big\{y \in M : Ty \ge z_\varepsilon \big\}$ and such that $\norm{x-y_0} + \delta < \varepsilon$. 
	Now, let $y \in M$ such that $\norm{x-y} \le \delta$. 
	Then the vector
	\begin{align*}
		\tilde y \coloneqq (y_0 + y - x)^+ \le y
	\end{align*}
	is in $M$ (since $\norm{x-\tilde y} \le \norm{x-y_0} +\norm{x-y}$) and satisfies $\norm{y_0 - \tilde y} \le \delta$ 
	(since $0 \le y_0 - \tilde y = (x-y) \wedge y_0 \le x-y$). 
	Hence, $Ty \ge T\tilde y \ge z_\varepsilon$, where the latter inequality follows from the choice of $\delta$.
\end{proof}

Now we have all partial results at hand to obtain our main result, a new characterization of kernel operators.

\begin{theorem}
	\label{thm:kernel-characterization}
	Assume that $E$ has order continuous norm and contains a quasi-interior point.
	Let $F$ be order complete such that $F'$ contains a strictly positive and order continuous functional.
	For each $T\in \cL(E,F)_+$ the following are equivalent:
	\begin{enumerate}[\upshape (i)]
		\item The operator $T$ is a kernel operator.
		\item There exists an at most countable $H \subseteq F_+$ such that for all $x\in E_+$ one has $Tx=\sup \big([0,Tx]\cap \R_+H\big)$.
	\end{enumerate}
	If the equivalent assertions {\upshape(i)} and {\upshape(ii)} hold, $H$ can be chosen as a subset of $Q(q)$ for any weak unit $q$ of $(TE)^{\bot\bot}$.
\end{theorem}

Note that, under the conditions of the theorem, 
for every quasi-interior point of $z$ of $E_+$, $q\coloneqq Tz$ is always a weak unit of $(TE)^{\bot\bot}$.

\begin{proof}
	(i) $\Rightarrow$ (ii): 
	This implication follows from Theorem \ref{thm:kernel-approx-non-sep}.

	(ii) $\Rightarrow$ (i): 
	This implication follows from Theorem \ref{thm:q-cont-approx-kernel-operators}.
	
	The addendum is a consequence of Proposition~\ref{prop:quasi-units}, 
	applied in the Banach lattice $(TE)^{\bot\bot}$ to the set $T(X_+)$.
\end{proof}

\section{Kernel Operators on $L^p$-spaces}

Now we use our results from Section~\ref{sec:range-of-kernel-ops} to characterize kernel operators on $L^p$-spaces:

\begin{theorem}
	\label{thm:characterization-integral-operators}
	Let $(\Omega_1,\Sigma_1,\mu_1)$ and $(\Omega_2,\Sigma_2,\mu_2)$ be $\sigma$-finite measure spaces.
	For $p\in [1,\infty)$ and $q\in [1,\infty]$ we set $E\coloneqq L^p(\Omega_1, \Sigma_1,\mu_1)$ and $F\coloneqq L^q(\Omega_2, \Sigma_2 , \mu_2)$.
	Then for each $T\in \cL(E,F)_+$ the following are equivalent:
	\begin{enumerate}[\upshape (i)]
		\item $T$ is a kernel operator.
		
		\item There exists a second countable topology $\cT$ on $\Omega_2$ with Borel $\sigma$-algebra $\cB(\cT) \subseteq \Sigma_2$
	such that $Tf$ has a lower semi-continuous representative with values in $[0,\infty]$ for every $f\in E_+$.
	\end{enumerate}
\end{theorem}
\begin{proof} 
	First note that $E$ and $F$ are both order complete, see \cite[Thm IV.8.22 and Thm IV.8.23]{dunford1958}. 
	Moreover, $E$ has order continuous norm and any element of $E_+$ that is strictly positive almost everyhwere 
	(such an element exists since $(\Omega_1, \mu_1)$ is $\sigma$-finite) is a quasi-interior point of $E_+$. 
	Each element in the dual of $F$ (in case $q=\infty$ in the predual of $F$) is order continuous
	and as $(\Omega_2, \mu_2)$ is $\sigma$-finite, there exists such a functional that is, in addition, strictly positive.
	Thus, the assumptions of Theorem \ref{thm:kernel-characterization} are satisfied.

	(i) $\Rightarrow$ (ii): Let $T\colon E \to F$ be a kernel operator.
	Since $(\Omega_2,\Sigma_2)$ is $\sigma$-finite, there exists a sequence $(F_n)\subseteq \Sigma_2$ of disjoint sets of finite measure 
	such that $\cup F_n = \Omega_2$. Then 
	\[g \coloneqq \sum_{n\in\N} 2^{-n} \mu(F_n)^{-1/q} \mathds{1}_{F_n} \]
	is a representative of a weak unit of $F$ (where we set $1/\infty \coloneqq 0$). 
	Hence, by Theorem \ref{thm:kernel-characterization} 
	there exists an at most countable set of quasi-units $U\subseteq Q(g)$ such that for all $f\in E_+$
	\begin{align}
		\label{eqn:Tfrepresentation}
		Tf = \sup \big([0,Tf]\cap \R_+ U \big),
	\end{align}
	where the supremum is taken within the Banach lattice $F$.
	We fix positive and measurable representatives of the elements of $U$.
	Then, after replacing each $h\in U$ by its components $h\mathds{1}_{F_1},h\mathds{1}_{F_2},\dots$ with respect to the exhausting sequence $(F_n)$ 
	and after rescaling we may assume that $U$ consists of indicator functions $U=\{ \mathds{1}_{B_n} : B_n\in \Sigma_2\text{, }n\in\N\}$.
		
	Now we consider the countable collection of measurable sets $(B_n)\subseteq \Sigma_2$ as subbasis of the topology
	\[\cT \coloneqq \bigcap \{ \tau \subseteq \cP(\Omega_2) : \tau \text{ is a topology and }B_n\in \tau \text{ for all }n\in\N\}.\]
	In other words,
	\[ \cT = \big\{ \bigcup \cM : \cM \subseteq \{ B_{k_1} \cap \dots \cap B_{k_n} : k_1,\dots,k_n \in \N\}\big\}\cup \{\Omega_2\}. \]
	This shows that $\cT$ is second countable and that $\cT \subseteq \Sigma_2$;
	hence, the Borel $\sigma$-algebra is coarser than $\Sigma_2$.

	Finally, fix $f \in E_+$. 
	Since $\R_+$ can be replaced with $\Q_+$ under the supremum in formula~\eqref{eqn:Tfrepresentation} and the supremum then coincides 
	with the $\mu_2$-almost everywhere supremum, we conclude that there exists a sequence $(\lambda_n) \subseteq \Q_+$ such that the pointwise supremum
	\[ \bigvee_{n\in\N} \lambda_n  \mathds{1}_{B_n} \colon \Omega_2 \to [0,\infty]\]
	is a representative of $Tf$. 
	This representative is lower semi-continuous with respect to $\cT$ since it is a pointwise supremum of the continuous functions $\mathds{1}_{B_n}$.

	(ii) $\Rightarrow$ (i): 
	Let $(B_n) \subseteq \Omega_2$ be a countable basis of $\cT$. 
	It follows immediately from the definitions that on the second countable space $(\Omega_2,\cT)$ every lower semi-continuous
	function $g \colon\Omega_2 \to [0,\infty]$ satisfies
	\[ g(x) = \bigvee_{n\in\N} \inf \{g(y) : y\in B_n\} \, \mathds{1}_{B_n} (x) \]
	for all $x\in \Omega_2$. 
	Hence, with $U\coloneqq \{\mathds{1}_{B_n} : n\in \N\}$, for every $f\in E_+$ the formula
	\[ Tf = \sup \big([0,Tf]\cap \R_+ U \big) \]
	holds in the Banach lattice $F$.
	So $T$ satisfies the assumptions of assertion (ii) of Theorem~\ref{thm:kernel-characterization} 
	and is thus a kernel operator.
\end{proof}

\section{Characterization of Partial Kernel Operators}
\label{sec:partialkernels}

As in Section~\ref{sec:range-of-kernel-ops}, let $E$ and $F$ be Banach lattices throughout this section. 
In what follows, we derive a characterization of partial kernel operators analogously to Theorem \ref{thm:kernel-characterization}.
We start with one of both implications which holds under slightly less restrictive assumptions than the characterization and which is reminiscent of Theorem~\ref{thm:q-cont-approx-kernel-operators}.

\begin{theorem}
	\label{thm:lower-bounds-partial-kernel}
	Let $E$ have order continuous norm, $F$ be order complete, and $H \subseteq F_+ \setminus\{0\}$ be an at most countable set.
	Let $T\in \cL(E,F)_+$ such that for every $0<x\in E$ there exists $h\in H$ such that $Tx \geq h$.
	Then there exists a kernel operator $K \in \cL(E,F)$ which satisfies $0 \leq K \leq T$ and which is strictly positive 
	in the sense that $Kx > 0$ for all $0 < x \in E$.
\end{theorem}
\begin{proof}
	Let $K$ denote the largest kernel operator between $0$ and $T$ and fix $0<x\in E$;
	this operator exists since the kernel operators form a band, and thus a projection band, within the order complete vector lattice $\regOps(E,F)$.
	We need to show that $Kx > 0$.
	
	For each $h\in H$ the set $A_h \coloneqq T^{-1}\{ v\in F_+ : v\geq h\}$ is closed and
	and we have, by assumption, 
	\[ [0,x] \subseteq \bigcup_{h\in H} A_h \cup \{0\}.\]
	So by Baire's category theorem there exists $h\in H$ such that $A_h\cap [0,x]$ has non-empty interior within $[0,x]$.
	Hence, we find $0<y_0 \in [0,x]$ and $ \delta > 0$ such that 
	$T\tilde y \geq h$ for all $\tilde y \in [0,x]$ with $\norm{y_0 - \tilde y} \leq \delta$.
	
	Now, consider an arbitrary vector $y\in [0,x]$ with $\norm{x-y} \leq \delta$; we claim that $Ty \geq h$. 
	To see this, set $\tilde y \coloneqq (y_0 + y - x)^+ \in [0, y] \subseteq [0,x]$ and observe that
	$\norm{y_0 - \tilde y} \leq \delta$ since $0 \leq y_0 - \tilde y = (x - y) \land y_0 \leq x-y$.
	Then indeed $Ty \geq T \tilde y \geq h$.
	
	Since $E$ has order continuous norm, we find by \cite[Thm~4.15]{aliprantis19852006} a functional $\varphi \in E'_+$ which is strictly positive on $[0,x]$. 
	Since the norm on $E$ is order continuous, the norm topology and the topology induced by the seminorm 
	$\applied{\varphi}{\abs{\argument}}$ coincide on $[0,x]$ \cite[Lem~4.16]{aliprantis19852006}. 
	Hence, by rescaling $\varphi$ we can achieve that 
	for all $y\in [0,x]$ with $\applied{\varphi}{x-y} \leq 1$ one has $\norm{x-y} \leq \delta$ and thus, as just seen, $Ty \geq h$. 
	Therefore,
	\begin{align*}
		Kx \ge (T\wedge (\varphi \otimes h))x = \inf \{ Ty + \applied{\varphi}{x-y} h : y \in [0,x] \} \geq h > 0.
	\end{align*}
	So $Kx > 0$ as claimed.
\end{proof}

As a consequence, one has the following characterization of partial kernel operators in the spirit of Theorem~\ref{thm:kernel-characterization}.

\begin{theorem}
	\label{thm:partial-kernel-characterization}
	Assume that $E = \{0\}$ has order continuous norm and that $E_+$ contains a quasi-interior point.
	Let $F$ be order complete such that $F'$ contains a strictly positive and order continuous functional.
	For every $T\in \cL(E,F)_+$ the following are equivalent:
	\begin{enumerate}[\upshape (i)]
		\item 
		$T$ is a partial kernel operator with a strictly positive kernel, 
		i.e.\ $T \geq  K \geq 0$ for some kernel operator $K$ such that $Kx>0$ for all $0<x\in E$. 
		
		\item 
		There exists an at most countable $H\subseteq F_+\setminus\{0\}$ with the following property: 
		for all $0<x\in E$ there exists $h\in H$ such that $Tx \geq h$.
	\end{enumerate}
\end{theorem}
\begin{proof}
	(i) $\Rightarrow$ (ii): 
	This follows from Theorem~\ref{thm:kernel-characterization} applied to the kernel operator $K$.
	As $E \not= \{0\}$ and $K$ is strictly positive, 
	for each $0 < x \in E$ one always finds a non-zero element $h \in H$ that satisfies $Kx \geq h$.
	
	(ii) $\Rightarrow$ (i): 
	This was shown in Theorem~\ref{thm:lower-bounds-partial-kernel}.
\end{proof}

It is worth noting that, under the conditions of Theorem \ref{thm:kernel-characterization} (or \ref{thm:partial-kernel-characterization}) on the spaces $E$ and $F$, 
the following holds: 
if $T,K \in \cL(E,F)_+$ such that $K$ is a kernel operator (or partial kernel operator with strictly positive kernel) 
and $TE_+ \subseteq KE_+$, then $T$ is a kernel operator (or partial kernel operator with strictly positive kernel).

It is desirable to also have a version of Theorem~\ref{thm:partial-kernel-characterization} available that works 
for operators which are positive but not strictly positive.  A simple result of this type is the following corollary.

\begin{corollary}
	\label{cor:partial-kernel-characterization-ideal}
	Assume that $E$ has order continuous norm and that $E_+$ contains a quasi-interior point.
	Let $F$ be order complete such that $F'$ contains a strictly positive and order continuous functional and
	let $J\subseteq E$ be a non-zero ideal. Then for every $T\in \cL(E,F)_+$ the following are equivalent:
	\begin{enumerate}[\upshape (i)]
		\item 
		$T$ is a partial kernel operator whose kernel is strictly positive on $J$,
		i.e.\ $T \geq  K$ for some kernel operator $K$ that satisfies $Kx>0$ for all $0<x\in J$.
		
		\item 
		There exists an at most countable set $H\subseteq F_+\setminus\{0\}$ with the following property: 
		for all $0<x\in J$ there exists $h\in H$ such that $Tx \geq h$.
	\end{enumerate}
\end{corollary}
\begin{proof}
	This follows from Theorem \ref{thm:partial-kernel-characterization} applied to $T|_B$, 
	where $B \subseteq E$ denotes the band generated by $J$.
\end{proof}

We close the article with a reformulation of Theorem~\ref{thm:partial-kernel-characterization} on $L^p$-spaces.

\begin{corollary}
	\label{cor:characterization-partial-integral-operators}
	Let $(\Omega_1,\Sigma_1,\mu_1)$ with $\mu_1 \not= 0$ and $(\Omega_2,\Sigma_2,\mu_2)$ be $\sigma$-finite measure spaces.
	For $p\in [1,\infty)$ and $q\in [1,\infty]$ 
	we set $E\coloneqq L^p(\Omega_1, \Sigma_1,\mu_1)$ and $F\coloneqq L^q(\Omega_2, \Sigma_2 , \mu_2)$.
	For every operator $T\in \cL(E,F)_+$ the following are equivalent:
	\begin{enumerate}[\upshape (i)]
		\item 
		$T$ is a partial kernel operator with a strictly positive integral kernel, 
		i.e.\ $T \geq K$ for some kernel operator $K$ such that $Kf>0$ for all $0<f\in E$. 
		
		\item 
		There exists a second countable topology $\cT$ on $\Omega_2$ with Borel $\sigma$-algebra $\cB(\cT) \subseteq \Sigma_2$
		such that for every $0< f \in E$ there exists $0<g \in F$ which satisfies $Tf \geq g$ 
		and which has a lower semi-continuous representative with values in $[0,\infty]$.
	\end{enumerate}
\end{corollary}

\begin{proof}
	(i) $\Rightarrow$ (ii): This follows from Theorem~\ref{thm:characterization-integral-operators}.

	(ii) $\Rightarrow$ (i): 
	Let $\cB$ denote a countable basis of the topology $\cT$.
	For every $0<f \in E$ it follows that
	$Tf \geq \lambda\mathds{1}_{B}$ for some $\lambda \in \Q_+\setminus\{0\}$ and some $B\in \cB$. 
	Therefore, Theorem~\ref{thm:partial-kernel-characterization} applied to $H\coloneqq \{ \lambda \mathds{1}_B : \lambda\in \Q_+\setminus\{0\}\text{, }B\in \cB\}$
	implies that $T$ is a partial kernel operator with strictly positive integral kernel.
\end{proof}

\subsection*{Acknowledgements}

We are grateful to the referee for a very useful 
hint on how to drop a separability assumption on $E$ throughout the article 
and for pointing out the characterization of kernel operators by means of equimeasurable sets 
in the articles \cite{grobler1985, schachermayer1981, schep1981}.

\bibliographystyle{amsplain}
\bibliography{kernel-range.bib}

\end{document}